\documentclass[A4,twoside,11pt,leqno]{article}
\usepackage{amssymb}
\usepackage{amsthm}
\usepackage[tbtags]{amsmath}
\usepackage{doc}
\usepackage{latexsym}
\usepackage{amscd}
\usepackage[frame,cmtip,arrow,matrix,line,graph,curve]{xy}
\usepackage{graphics}
\usepackage{epsfig}
\usepackage{eucal}
\usepackage{mathrsfs}
\usepackage{amsmath,amssymb,amsfonts,enumerate,amsthm, amscd}
\usepackage[french, ngerman, italian, english]{babel}
\usepackage[colorlinks=true]{hyperref}
\usepackage{paralist}
\usepackage{verbatim}
\usepackage{color}
\usepackage{xcolor}
\usepackage{tikz}
\usetikzlibrary{arrows,decorations.pathmorphing,backgrounds,positioning,fit}
\RequirePackage{pgfrcs}
\theoremstyle{plain}
\newtheorem{thm}{\bf Theorem}[section]

\newtheorem{lem}[thm]{\bf Lemma}
\newtheorem{cor}[thm]{\bf Corollary}

\theoremstyle{definition}
\newtheorem{defn}[thm]{\bf Definition}
\theoremstyle{remark}
\newtheorem{rem}[thm]{\bf Remark}
\newtheorem{exam}[thm]{\bf Example}

\theoremstyle{example}

\def \TE{\mathrm{TE}}

\def \1{\mathbf 1}

\def \NN{\mathbb N}

\def \P{\mathcal P}

\def \E{\mathcal E}

\def \E{\mathcal E}

\def \A{\mathcal A}
\def \C{\mathcal C}
\def \M{\mathcal M}
\def \B{\mathcal B}

\font\headd=cmr8
\begin{document}
\thispagestyle{plain}
 \markboth{}{}
\small{\addtocounter{page}{0} \pagestyle{plain}
\noindent{\scriptsize }\\
\noindent{\scriptsize }\\
\noindent {\scriptsize }\\
\noindent {\scriptsize }
\vspace{0.2in}\\
\noindent{\large\bf Expansion and contraction functors on matriods}
\footnote{{}\\ \\[-0.7cm]
* Corresponding Author.
}\vspace{0.15in}\\
\noindent{\sc Rahim Rahmati-Asghar$^*$}
\newline
{\it Department of Mathematics, Faculty of Basic Sciences, University of Maragheh, P. O. Box 55181-83111, Maragheh, Iran.\\
e-mail} : {\verb|rahmatiasghar.r@gmail.com|}
\vspace{0.15in}\\
{\footnotesize {\sc Abstract.} Let $M$ be a matroid. We study the expansions of $M$ mainly to see how the combinatorial properties of $M$ and its expansions are related to each other. It is shown that $M$ is a graphic, binary or a transversal matroid if and only if an arbitrary expansion of $M$ has the same property. Then we introduce a new functor, called contraction, which acts in contrast to expansion functor. As a main result of paper, we prove that a matroid $M$ satisfies White's conjecture if and only if an arbitrary expansion of $M$ does. It follows that it suffices to focus on the contraction of a given matroid for checking whether the matroid satisfies White's conjecture. Finally, some classes of matroids satisfying White's conjecture are presented.
}
\vspace{0.2in}\\
\pagestyle{myheadings}
 \markboth{\headd Rahim Rahmati-Asghar$~~~~~~~~~~~~~~~~~~~~~~~~~~~~~~~~~~~~~~~~~~~~~\,$}
 {\headd $~~~~~~~~~~~~~~~~~~~~~~~~~~~~~~~~~~~~~$Expansion and contraction functors on matriods}\\
%
%
%
%
\section*{Introduction}

Matroids are abstract combinatorial structures that capture the notion of independence that is common to a surprisingly large number of mathematical
entities. They were introduced by Whitney in 1935 as a common generalization of independence in linear algebra and independence in graph theory \cite{Whn}. Matroid theory is one of the most fascinating research areas in combinatorics. It was linked to projective geometry by Mac Lane \cite{Ma}, and have found a great many applications in several branches of mathematics \cite{Wh}. In this regard studying the structural properties of matroids from different point views have been considered by many researchers. Also classifying matroids with a desired property or making modifications to a matroid so that it satisfies a special property is the subject of many research papers, see for example \cite{Bl,Ho,Ka,LaMi,RaYa,Sc,Wh}.

The notion of expansion is a known notion appeared in different terminologies as parallelization or duplication in combinatorics \cite{Go,FrHaVi,MaReVi,Scr}. In \cite{RaYa}, the authors studied behaviors of expansion functor on some algebraic structures associated to discrete polymatroids. It was shown that a nonempty finite set is a discrete polymatroid if and only if its an arbitrary expansion is a discrete polymatroid (c.f. \cite[Theorem 1.2.]{RaYa}). The discrete polymatroid is a multiset analogue of the matroid. Moreover, there are several classes of matroids so that the study of each of them is interesting in its own right. This motivates to focus on the behaviors of the expansion functor on some classes of matroids and to investigate some structural properties of them in this paper. Our goal in this paper is to investigate more relations between the exchange property of bases of a matroid and those of its expansions. It turns out that the exchange properties of bases of a matroid are preserved under taking the expansion functor
and so this construction is a very good tool to make new matroids with a desired property. Moreover, by taking another functor, contraction functor, which is the opposite to the expansion functor we will able to construct a new matroid, with possibly smaller ground set, from a given matroid and check a desired property on new matroid instead of primary one.

White in 1980 proposed a conjecture about the bases of a matroid \cite{Wh}. This conjecture has received much attention in recent years and has some algebraic and combinatorial variants, all of which are open problems. Up to now, several mathematicians confirmed only some variants of this conjecture for special classes of matroids (see for example \cite{Bl, Bo, Ho, Ka, LaMi, RaYa, Sc, St}).

White \cite{Wh} defined three classes $\TE(1)$, $\TE(2)$ and $\TE(3)$ of matroids and conjectured that $\TE(1)=\TE(2)=\TE(3)=$the class of all matroids.

We investigate the effect of the expansion functor on the exchange property for bases of matroids and conclude that White's conjecture is preserved under taking the expansion or contraction functor.

The paper is organized as follows. In Section 1, we review some preliminaries which are needed in the sequel. In Section 2, we investigate the expansion of some classes of matroids. We show that a matroid is graphic, binary or transversal if and only if its an arbitrary expansion has such a property (see Theorems \ref{graphic}, \ref{binary} and \ref{transversal}). Also, we prove that the expansion of an uniform matroid is a partition matroid and, conversely, every partition matroid is an expansion of an uniform matroid (see Theorem \ref{partition uniform}). In Section 3, we introduce the contraction functor which acts in contrast to expansion functor. The last section is devoted to the study of unique exchange property. After recalling some notions and notations from \cite{Wh}, we formulate White's conjecture \cite[Conjecture 12]{Wh}. As one of the main results, we show that a matroid $M$ satisfies White's conjecture if and only if an arbitrary expansion of $M$ does (see Theorem \ref{TE}). This concludes that $M$ satisfies White's conjecture if and only if its contraction does (see Corollary \ref{contract white}). On the other hand, since the class of contracted matroids is very smaller than the class of all matroids, it follows from Corollary \ref{contract white} that to test White's conjecture for a given class of matroids it suffices to turn our attention to their contractions. Finally, we give some classes of matroids which satisfy White's conjecture.

\section{Preliminaries}

A matroid $M$ is a pair $(\E_M,\B_M)$ consisting of a finite set $\E_M$
and a non-empty family $\B_M$ of subsets of $\E_M$ such that no set in $\B_M$ properly contains another set in $\B_M$ and, moreover, $\B_M$ satisfies the following \emph{exchange property}:
\begin{center}
 for every $B_1,B_2\in\B_M$ and $x\in B_1\backslash B_2$ there exists $y\in B_2\backslash B_1$, such that $(B_1\backslash x)\cup y\in\B_M$.
\end{center}
$\E_M$ and $\B_M$ are, respectively, called the \emph{ground set} and the \emph{basis set} of $M$.
The background from matroid theory which we use may be obtained from \cite{Ox} or \cite{We}.

Recall from \cite{RaYa} the concept of expansion functor on a family of subsets of $[n]=\{x_1,\ldots,x_n\}$. Let $\alpha=(k_1,\ldots,k_n)\in\NN^n$. For $A=\{x_{i_1},\ldots,x_{i_r}\}\subseteq [n]$, the \emph{expansion} of $A$ is defined
$$A^\alpha=\{x_{i_11},\ldots,x_{i_1k_{i_1}},\ldots,x_{i_r1},\ldots,x_{i_rk_{i_r}}\}.$$

Let $\A$ be a family of subsets of $[n]$ and $A=\{x_{i_1},\ldots ,x_{i_r}\}\in\A$. Set $[n]^\alpha=\{x_{ij}:1\leq i\leq n,1\leq j\leq k_i\}$. The \emph{expansion} of the singleton family $\{A\}$ with respect to $\alpha$ is denoted by $\{A\}^\alpha$ and it is a family of subsets of $[n]^\alpha$ defined as follows:
$$\{A\}^\alpha=\{\{x_{i_1j_1},\ldots ,x_{i_rj_r}\}:1\leq j_l\leq k_{i_l}\ \mbox{for all}\ l\}.$$
Also, the expansion of $\A$ with respect to $\alpha$ is denoted by $\A^\alpha$ and it is defined
$$\A^\alpha=\bigcup_{A\in\A}\{A\}^\alpha.$$

Let $2^{[n]}$ denote the set of all subsets of $[n]$ and let $\alpha\in\NN^n$. We define the map
$\pi:2^{[n]^\alpha}\rightarrow 2^{[n]}$ by setting $\pi(\{x_{i_1j_1},\ldots,x_{i_rj_r}\})=\{x_{i_1},\ldots,x_{i_r}\}$.

The following theorem is a direct consequence of \cite[Theorem 1.2]{RaYa}:

\begin{thm}\label{expan matroid}
Let $M$ be a nonempty family of subsets of $[n]$ and let $\alpha\in\NN^n$. Then $M$ is a matroid if and only if $M^\alpha$ is.
\end{thm}

The \emph{restriction} of a matroid $(\E_M,\B_M)$ to $X\subseteq\E_M$ is denoted by $M_X$ and it is a matroid with the ground set $\E_{M_X}=\E_M\cap X$ and the basis set
$$\B_{M_X}=\{B\cap X:B\in\B_M\}.$$

\section{The expansion of some classes of matroids}

Let $G=(V(G),E(G))$ be an undirected graph (with possibly loops or parallel edges). A \emph{spanning subgraph} of a graph $G$ is a subgraph whose vertex set is the entire vertex set of $G$. If this spanning subgraph is a tree, it is called a \emph{spanning tree} of the graph.

Let $\E=E(G)$ and $\B=\{B\subset E(G):B\ \mbox{is a spanning tree of}\ G\}$. Then $M=(\E,\B)$ is a matroid called a \emph{graphic matroid}. The graphic matroid associated with the graph $G$ is denoted by $M(G)$.

\begin{thm}\label{graphic}
Let $M$ be a matroid on $[n]$ and $\alpha=(k_1,\ldots,k_n)\in\NN^n$. Then $M$ is graphic if and only if $M^\alpha$ is graphic.
\end{thm}

We need some notations and an auxiliary lemma:

For any $1\leq i\leq n$, let $\varepsilon_i=(a_1,\ldots,a_n)\in \NN^n$ be defined as
$a_j=\left\{
\begin{array}{ll}
0  & \hbox{if}\ j\neq i \\
1 & \hbox{if}\ j=i.
\end{array}
\right. $
Set $\1=(1,1,\ldots,1)\in\NN^n$.

\begin{lem}\label{epsilon}
Let $\A$ be a family of subsets of $[n]$, $\beta=(k_1,\ldots,k_n)\in\NN^n$ and $\alpha=\beta+\delta_i$. Then $(\A^\beta)^{\1+\varepsilon_{ik_i}}\cong \A^\alpha$.
\end{lem}
\begin{proof}
Note that
$$[n]^\alpha=\{x_{11},\ldots,x_{1k_1},\ldots,x_{i1},\ldots,x_{ik_i},x_{i(k_i+1)},\ldots,x_{n1},\ldots,x_{nk_n}\}$$
and
$$([n]^\beta)^{\1+\varepsilon_{ik_i}}=\{x_{111},\ldots,x_{1k_11},\ldots,x_{i11},\ldots,x_{ik_i1},x_{ik_i2},\ldots,x_{n11},\ldots,x_{nk_n1}\}.$$

Define $\varphi:([n]^\beta)^{\1+\varepsilon_{ik_i}}\rightarrow [n]^\alpha$ given by
$$\varphi(x_{rst})=\left\{
  \begin{array}{ll}
    x_{rs} & t=1 \\
    x_{i(k_i+1)} & t=2.
  \end{array}
\right.$$

Then $\varphi$ induces the bijection

$$\begin{array}{cr}
    \theta: & (\A^\beta)^{\1+\varepsilon_{ik_i}}\rightarrow \A^\alpha. \\
     & F\mapsto \varphi(F)
  \end{array}$$
\end{proof}

Now we prove Theorem \ref{graphic}:

\begin{proof}
Let $G$ be a graph with $E(G)=\{x_1,\ldots,x_n\}$ and $M\cong M(G)$. We use induction on $\alpha$.

First, suppose that $k_1=2$ and $k_i=1$ for all $i>1$. Let $G'$ be a graph with the vertex set $V(G)$ and the edge set $E(G')=\{x_{i1},x_{12}:i=1,\ldots,n\}$ where $x_{i1}=x_i$ for all $i$ and $x_{11}$ and $e_{12}$ are parallel. It is easy to check that $M(G)^\alpha\cong M(G')$.

Now, suppose that $k_1>1$ and $\beta\in\NN^n$ with $\beta(i)=k_i$ if $2\leq i\leq n$ and $\beta(1)=k_1-1$. Assume that $M(G)^\beta\cong M(H)$ is graphic. Then it follows from induction hypothesis and Lemma \ref{epsilon} that $M(G)^\alpha\cong (M(G)^\beta)^{\varepsilon_1}\cong M(G')$ where $G'$ is obtained from $H$ by adding a parallel edge.

Conversely, suppose $M^\alpha$ is graphic and $M^\alpha\cong M(G')$. Set $X=\{x_{i1}:1\leq i\leq n\}$. Then $(M^\alpha)_X\cong M$ is graphic by \cite[page 842]{Jo}.
\end{proof}

\begin{exam}
Consider the matroid $M$ on $[6]$ with basis set
$$\B_M=\{\{x_1,x_2,x_3\},\{x_1,x_2,x_4\},\{x_1,x_3,x_4\},\{x_1,x_3,x_5\},$$
$$\{x_1,x_4,x_5\},\{x_2,x_3,x_4\},\{x_2,x_3,x_5\},\{x_2,x_4,x_5\}\}.$$
$M$ is a graphic matroid associated with the graph $G$ shown in Figure \ref{2}. Let $\alpha=(1,1,1,1,2,2)\in\NN^6$. Then $M^\alpha$ is a graphic matroid and $M^\alpha$ is associated with $G'$ (see Figure \ref{2}), obtained from $G$ by adding parallel edges to $x_5$ and $x_6$.
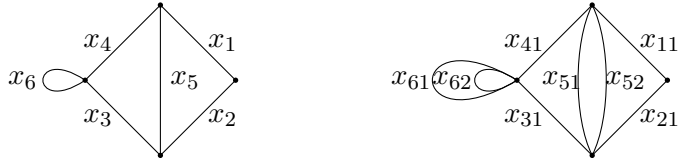
\begin{figure}
$$\begin{array}{ccccc}
\begin{tikzpicture}
\coordinate (v1) at (0,1);\fill (0,1) circle (1pt);
\coordinate (v2) at (1,0);\fill (1,0) circle (1pt);
\coordinate (v3) at (0,-1);\fill (0,-1) circle (1pt);
\coordinate (v4) at (-1,0);\fill (-1,0) circle (1pt);
\node[left] at (-1.5,0)    {$x_6$};
\draw (v1) to node[right] {$x_1$} (v2);
\draw (v2) to node[right] {$x_2$} (v3);
\draw (v3) to node[left] {$x_3$} (v4);
\draw (v1) to node[left] {$x_4$} (v4);
\draw (v1) to node[right] {$x_5$} (v3);
\draw (-1,0) .. controls (-1.75,0.5) and (-1.75,-0.5) .. (-1,0);
\end{tikzpicture}
& \ & \ & \ &
\begin{tikzpicture}
\coordinate (v1) at (0,1);\fill (0,1) circle (1pt);
\coordinate (v2) at (1,0);\fill (1,0) circle (1pt);
\coordinate (v3) at (0,-1);\fill (0,-1) circle (1pt);
\coordinate (v4) at (-1,0);\fill (-1,0) circle (1pt);
\node[left] at (-1.45,0)    {$x_{62}$};
\node[left] at (-2,0)    {$x_{61}$};
\node[left] at (0.85,0)    {$x_{52}$};
\node[right] at (-0.8,0)    {$x_{51}$};
\draw (v1) to node[right] {$x_{11}$} (v2);
\draw (v2) to node[right] {$x_{21}$} (v3);
\draw (v3) to node[left] {$x_{31}$} (v4);
\draw (v1) to node[left] {$x_{41}$} (v4);
\draw (0,1) .. controls (-0.25,0.5) and (-0.25,-0.5) .. (0,-1);
\draw (0,1) .. controls (0.25,0.5) and (0.25,-0.5) .. (0,-1);
\draw (-1,0) .. controls (-1.75,0.5) and (-1.75,-0.5) .. (-1,0);
\draw (-1,0) .. controls (-2.5,0.9) and (-2.5,-0.9) .. (-1,0);
\end{tikzpicture}
\end{array}$$
  \caption{The graphs $G$ and $G'$}\label{2}
\end{figure}

\end{exam}

A subset of the ground set of a matroid $M$ that is contained in no bases of $M$ is called \emph{dependent}. A \emph{circuit} in $M$ is a minimal dependent subset (with respect to inclusion) of $\E_M$ and the set of circuits of $M$ is denoted by $\C(M)$. A matroid $M$ is \emph{binary} if and only if for every pair of circuits of $M$, their symmetric difference contains another circuit. See \cite[Theorem 9.1.2]{Ox} for other equivalent definitions of binary matroids.

\begin{thm}\label{binary}
Let $M$ be a matroid on $[n]$ and $\alpha\in\NN^n$. Then $M$ is binary if and only if $M^\alpha$ is.
\end{thm}
\begin{proof}
Let $M$ be binary and let $C'_1,C'_2\in\C(M^\alpha)$. If $\pi(C'_1)=\pi(C'_2)$ then there exist $x_{ij}\in C'_1\backslash C'_2$ and $x_{ij'}\in C'_2\backslash C'_1$ with $j\neq j'$. Set $C'=\{x_{ij},x_{ij'}\}$. Then $C'\subset C'_1\triangle C'_2$ and $C'\in\C(M^\alpha)$, and hence the assertion is completed. So suppose that $\pi(C'_1)\neq\pi(C'_2)$.

If $|\pi(C'_1)|=|C'_1|$ and $|\pi(C'_2)|=|C_2|$ then since $\pi(C'_1),\pi(C'_2)\in\C(M)$ there exists $C\in\C(M)$ such that $B\subseteq \pi(C'_1)\triangle \pi(C'_2)$. Since $\pi(C'_1)\triangle \pi(C'_2)\subseteq \pi(C'_1\triangle C'_2)$ we have $B\subseteq \pi(C'_1\triangle C'_2)$ and so it follows that $C'\subseteq C'_1\triangle C'_2$ for some $C'\in\C(\M^\alpha)$ with $\pi(C')=C$.

Consider $|\pi(C'_1)|\neq|C'_1|$ or $|\pi(C'_2)|\neq|C'_2|$. Let, for example, $|\pi(C'_1)|\neq|C'_1|$. Then $\{x_{il},x_{im}\}\subset C'_1$ for some $l,m$. Since $\{x_{il},x_{im}\}\in\C(M^\alpha)$, thus it should be $C'_1=\{x_{il},x_{im}\}$. Consider the following cases:
\begin{itemize}
  \item $|\pi(C'_2)|\neq|C'_2|$: Then $C'_2=\{x_{jl'},x_{jm'}\}$ for some $j,l'$ and $m'$. Since $\pi(C'_1)\neq\pi(C'_2)$ we have $i\neq j$. Set $C'=\{x_{il},x_{im}\}$.
  \item $|\pi(C'_2)|=|C'_2|$: Then $C'_2=\{x_{i_1j_1},\ldots,x_{i_rj_r}\}$ and $C'_1\triangle C'_2=\{x_{i_1j_1},\ldots,x_{i_rj_r},x_{il},x_{im}\}$ or $C'_1\triangle C'_2=\{x_{i_1j_1},\ldots,x_{i_{r-1}j_{r-1}},x_{il}\}$ where $x_{i_rj_r}=x_{im}$. At the first case, set $C'=\{x_{il},x_{im}\}$ and at the second one, set $C'=\{x_{i_1j_1},\ldots,x_{i_{r-1}j_{r-1}},x_{il}\}$.
\end{itemize}
Thus $C'\in\C(M^\alpha)$ and $C'\subset C'_1\triangle C'_2$. Therefore $M^\alpha$ is binary.

Conversely, suppose that $M^\alpha$ is binary and $C_1,C_2\in\C(M)$. Let $C_1=\{x_{i_1},\ldots,x_{i_r}\}$ and $C_2=\{x_{j_1},\ldots,x_{j_s}\}$. Set $C'_1=\{x_{i_11},\ldots,x_{i_r1}\}$ and $C'_2=\{x_{j_11},\ldots,x_{j_s1}\}$. Then $C'_1,C'_2\in\C(M^\alpha)$ and so there exists $C'\in\C(M^\alpha)$ with $C'\subset C'_1\triangle C'_2$. It follows that $\pi(C')\in\C(M)$ and $\pi(C')\subset C_1\triangle C_2$. Therefore $M$ is binary.
\end{proof}

We recall from \cite[page 46.]{Ox} the definition of a transversal matroid. Let $S\subset [n]$. A \emph{set system} $(S,\A)$ is a set $S$ along with a multiset $\A=(A_j:j\in J)$ of (not necessarily distinct) subsets of $S$. If $J=\{1,\ldots,m\}$ then we may denote $\A$ by $\A=(A_1,\ldots,A_m)$. A \emph{transversal} of $\A=(A_j:j\in J)$ is a subset $T$ of $S$ for which there is a bijection $\varphi:J\rightarrow T$ with $\varphi(j)=A_j$ for all $j\in J$.

\begin{thm}\label{Hall}
\cite{Ha} A finite set system $(S, (A_j:j\in J))$ has a transversal if and only if, for all $K\subset J$,
$$|\bigcup A_i|\geq |K|.$$
\end{thm}

If $X\subset S$, then $X$ is a \emph{partial transversal} of $\A=(A_j:j\in J)$ if, for some subset $K$ of $J$, $X$ is a transversal of $\A$. The partial transversals of a $\A$ are the independent sets of a matroid. We call such a matroid a \emph{transversal matroid} and denote it by $M[\A]$. $\A$ is called a \emph{presentation} of $M[\A]$.

\begin{thm}\label{restric trans}
\cite{Br} Let $M$ be a transversal matroid on $[n]$. Then so is $M_X$ for each $X\subset [n]$. If $(A_1,\ldots,A_m)$ is a presentation of $M$, then $(A_1\cap X,\ldots,A_m\cap X)$ is a presentation of $M_X$.
\end{thm}

If $\A=(A_j:j\in J)$ is a family of subsets of $S\subset [n]$ then the bipartite graph associated with $\A$, denoted by $G[\A]$, has the vertex set $S\cup\{A_j:j\in J\}$ and the edge set $\{x_iA_j:j\in J\ \mbox{and}\ x_i\in A_j\}$.

A matching in a graph $G$ is the set of edges in $G$ no two of which have a common endpoint. A subset $X$ of $S$ is a partial transversal of $\A$ if and only if there is a matching in $G[\A]$ which every edge has one endpoint in $X$.

\begin{thm}\label{transversal}
Let $M$ be a matroid on $[n]$ and let $\alpha\in\NN^n$. Then $M$ is transversal if and only if $M^\alpha$ is.
\end{thm}
\begin{proof}
``Only if part'': Let $M$ be a transversal matroid with $M\cong M[\A]$ where $\A=(A_1,A_2,\ldots,A_m)$ and $A_i\subset [n]$. Set $\A^{\alpha}=(A^{\alpha}_1,A^{\alpha}_2,\ldots,A^{\alpha}_m)$. We claim that $\A^\alpha$ is a presentation of $M^\alpha$. We associate to $A_j$ and $A^\alpha_j$, respectively, the vertices $y_j$ and $y'_j$.

Let $B'\in\B_{M[\A]^\alpha}$. We may assume that $B'=\{x_{1i_1},\ldots,x_{ri_r}\}$. So there exists the maximal matching $\{x_1y_{j_1},\ldots,x_ry_{j_r}\}$ in the bipartite graph $G[\A]$ with the partition $[n]\dot{\cup}\{y_1,\ldots,y_m\}$. It is clear that $B''=\{x_{1i_1}y'_{j_1},\ldots,x_{ri_r}y'_{j_r}\}$ is a matching in $G[\A^\alpha]$. Suppose, on the contrary, that $B''$ is not maximal. So for some matching $C$ in $G[\A^\alpha]$ we have $B''\subset C$. Let $x_{st}y'_{l}\in C\backslash B''$. Since $\{x_1y_{j_1},\ldots,x_ry_{j_r}\}$ is maximal, we have $s\in\{1,\ldots,r\}$. Moreover, it is clear that $l\not\in\{j_1,\ldots,j_r\}$. Let $s=1$ and let $X=\{x_1,\ldots,x_r\}$. By Theorem \ref{restric trans}, $M_X$ is transversal with presentation $(A_1\cap X,\ldots,A_m\cap X)$. But $|\cup_{k\in\{j_1,\ldots,j_r,l\}}(A_k\cap X)|=r<r+1=|\{j_1,\ldots,j_r,l\}|$. This contradicts Theorem \ref{Hall}. Therefore $B''$ is a maximal matching in $G[\A^\alpha]$ and so $B'\in\B_{M[\A^\alpha]}$.

In a similar argument we show that $\B_{M[\A^\alpha]}\subseteq \B_{M[\A]^\alpha}$. Therefore $M[\A]^\alpha\cong M[\A^\alpha]$, as desired.

``If part'': Let $M^\alpha$ be transversal with the presentation $\B=(B_1,\ldots,B_m)$. One may suppose that $[n]:=\{x_{i1}:1\leq i\leq n\}$. By Theorem \ref{restric trans}, $(M^\alpha)_{[n]}$ is a transversal matroid with the presentation $\B'=(B_1\cap [n],\ldots,B_m\cap [n])$. Since $(M^\alpha)_{[n]}=M$, the assertion is completed.

\end{proof}

\begin{exam}
Let $S=\{x_1,x_2,x_3,x_4\}\subset [n]$ and $\alpha=(k_1,\ldots,k_n)\in\NN^n$. Let $\A=(\{x_1,x_2\},\{x_1,x_3\},\{x_3\})$. Then the matchings in $G[\A]$ are
$$\{x_1y_2,x_2y_1\},\{x_1y_1,x_3y_2\},\{x_2y_1,x_3y_2\}$$
and so
$$\B_{M[\A]}=\{\{x_1,x_2\},\{x_1,x_3\},\{x_2,x_3\}\}.$$
$G[\A]$ and $G[\A^\alpha]$ are shown in Figure \ref{1}.
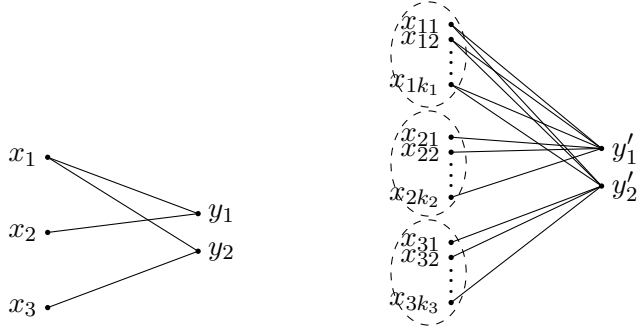
\begin{figure}
$$\begin{array}{ccccc}
\begin{tikzpicture}
\coordinate (x1) at (0,3);\fill (0,3) circle (1pt);
\node[left] at (0,3)    {$x_1$};
\coordinate (x2) at (0,2);\fill (0,2) circle (1pt);
\node[left] at (0,2)    {$x_2$};
\coordinate (xn) at (0,1);\fill (0,1) circle (1pt);
\node[left] at (0,1)    {$x_3$};
\coordinate (A1) at (2,2.25);\fill (2,2.25) circle (1pt);
\node[right] at (2,2.25)    {$y_1$};
\coordinate (A2) at (2,1.75);\fill (2,1.75) circle (1pt);
\node[right] at (2,1.75)    {$y_2$};
\draw (x1) -- (A1);
\draw (x1) -- (A2);
\draw (x2) -- (A1);
\draw (xn) -- (A2);
\end{tikzpicture}
& \ & \ & \ &
\begin{tikzpicture}
\coordinate (x11) at (0,3.4);\fill (0,3.4) circle (1pt);
\node[left] at (0,3.4)    {$x_{11}$};
\coordinate (x12) at (0,3.2);\fill (0,3.2) circle (1pt);
\node[left] at (0,3.2)    {$x_{12}$};
\node at (0,3)    {$\vdots$};
\coordinate (x1k1) at (0,2.6);\fill (0,2.6) circle (1pt);
\node[left] at (0,2.6)    {$x_{1k_1}$};
\draw[style=dashed]  (-0.3,3) ellipse (0.5 and 0.7);
\coordinate (x21) at (0,1.9);\fill (0,1.9) circle (1pt);
\node[left] at (0,1.9)    {$x_{21}$};
\coordinate (x22) at (0,1.7);\fill (0,1.7) circle (1pt);
\node[left] at (0,1.7)    {$x_{22}$};
\node at (0,1.5)    {$\vdots$};
\coordinate (x2k2) at (0,1.1);\fill (0,1.1) circle (1pt);
\node[left] at (0,1.1)    {$x_{2k_2}$};
\draw[style=dashed]  (-0.3,1.55) ellipse (0.5 and 0.7);
\coordinate (xn1) at (0,0.5);\fill (0,0.5) circle (1pt);
\node[left] at (0,0.5)    {$x_{31}$};
\coordinate (xn2) at (0,0.3);\fill (0,0.3) circle (1pt);
\node[left] at (0,0.3)    {$x_{32}$};
\node at (0,0.1)    {$\vdots$};
\coordinate (xnkn) at (0,-0.3);\fill (0,-0.3) circle (1pt);
\node[left] at (0,-0.3)    {$x_{3k_3}$};
\draw[style=dashed]  (-0.3,0.1) ellipse (0.5 and 0.7);

\coordinate (A1) at (2,1.75);\fill (2,1.75) circle (1pt);
\node[right] at (2,1.75)    {$y'_1$};
\coordinate (A2) at (2,1.25);\fill (2,1.25) circle (1pt);
\node[right] at (2,1.25)    {$y'_2$};

\draw (x11) -- (A1);
\draw (x12) -- (A1);
\draw (x1k1) -- (A1);
\draw (x11) -- (A2);
\draw (x12) -- (A2);
\draw (x1k1) -- (A2);
\draw (x21) -- (A1);
\draw (x22) -- (A1);
\draw (x2k2) -- (A1);
\draw (xn1) -- (A2);
\draw (xn2) -- (A2);
\draw (xnkn) -- (A2);
\end{tikzpicture}
\end{array}$$
  \caption{The bipartite graphs $G[\A]$ and $G[\A^\alpha]$}\label{1}
\end{figure}

\end{exam}

A matroid on $[n]$ of rank $t\leq n$ is an \emph{uniform} matroid if all $t$-element subsets of $[n]$ are bases and it is denoted by $U_{t,n}$.

A \emph{partition matroid} of rank $t$ \cite{KaOkUn} is a matroid $M(\P)$ associated with a partition $\P=\{A_1,\ldots,A_m\}$ of $\E_{M(\P)}$ and the basis set
$$\B_{M(\P)}=\{U\subset A:|U\cap A_i|\leq 1\ \mbox{for all}\ i\ \mbox{and}\ |U|=t\}.$$

\begin{thm}\label{partition uniform}
The expansion of every uniform matroid is a partition matroid. Conversely, every partition matroid is the expansion of an uniform matroid.
\end{thm}
\begin{proof}
Let $U_{t,n}$ be an uniform matroid on $[n]$ of rank $t$ and let $\alpha=(k_1,\ldots,k_n)\in\NN^n$. Set $A_i=\{x_{i1},\ldots,x_{ik_i}\}$ for all $i$. Then it is easy to see that $U^\alpha_{t,n}=M(\P)$ where $\P=\{A_1,\ldots,A_n\}$.

Conversely, let $M(\P)$ be a partition matroid of rank $t$ with $\P=\{A_1,\ldots,A_m\}$. Set $A'_i:=\{x_i\}$ for all $i$. Then $M(\P)\cong U^\alpha_{t,m}$ where $\alpha=(k_1,\ldots,k_m)$ and $k_i=|A_i|$ for all $i$.
\end{proof}

\section{The contraction functor}
\begin{defn}
Let $\A$ be a family of subsets of $[n]$ and let $\B$ denote the maximal elements of $\A$ (with respect to inclusion). We define the relation $``\sim$'' on $[n]$ in the following form:
$$x_i\sim x_j \Longleftrightarrow \{A\backslash x_i: A\in\B, x_i\in A\}=\{A\backslash x_j: A\in\B, x_j\in A\}.$$
In other words,
$$x_i\sim x_j \Longleftrightarrow \mbox{for all}\ A\in\B \left\{
                                    \begin{array}{ll}
                                      \mbox{if}\ x_i\in A\ \mbox{then}\ (A\backslash x_i)\cup x_j\in\B \\
                                      \mbox{if}\ x_j\in A\ \mbox{then}\ (A\backslash x_j)\cup x_i\in\B.
                                    \end{array}
                                  \right.
$$

It is easily shown that $\sim$ is an equivalence relation. Let $[m]=\{y_1,\ldots,y_m\}$ be the set of equivalence classes under $\sim$.

Let $y_i=\{x_{i1},\ldots,x_{ia_i}\}$. Set $\alpha=(a_1,\ldots,a_m)$. For $A\in\B$, define $\overline{A}=\{y_i:y_i\cap A\neq\emptyset\}$ and $\overline{\A}$ a family of subsets of $[m]$ with the set $\{\overline{A}:A\in\B\}$ of maximal elements. We call $\overline{\A}$ the \emph{contraction of $\A$ by $\alpha$}. Clearly, every family of subsets of $[n]$ has an unique contraction.

A family $\A$ of subsets of $[n]$ is called \emph{contracted} if $\A$ and $\overline{\A}$ coincide up to a relabeling.
\end{defn}

\begin{rem}\label{expan contract}
Note that the contraction functor behaves exactly the opposite to expansion functor. Actually, if $\A$ is a family of subsets of $[n]$ and $\overline{\A}$ is the contraction of $\A$ by $\alpha$, then $(\overline{\A})^\alpha$ and $\A$ coincide up to a relabeling of $[n]$. Also, for every $\alpha\in\NN^n$, two families $\overline{\A^\alpha}$ and $\overline{\A}$ coincide. Therefore $\A$ is a matroid if and only if $(\overline{\A})^\alpha$ is a matroid. Equivalently, by Theorem \ref{expan matroid}, $\overline{\A}$ is a matroid.
\end{rem}

\begin{rem}\label{uniform contracted}
All of uniform matroids of rank $t>1$ are contracted.
\end{rem}

\begin{cor}\label{cont part unif}
The contraction of a partition matroid is an uniform matroid.
\end{cor}
\begin{proof}
Let $M$ be a partition matroid on $[n]$. By Theorem \ref{partition uniform}, $M$ is the expansion of an uniform matroid $U_{t,n}$ with respect to some $\alpha\in\NN^n$. It follows from Remarks \ref{expan contract} and \ref{uniform contracted} that $\overline{M}=\overline{U^\alpha_{t,n}}=\overline{U_{t,n}}=U_{t,n}$.
\end{proof}

In view of Theorems \ref{graphic}, \ref{binary}, \ref{transversal} and \ref{partition uniform} we have the following:

\begin{cor}
The contraction of a graphic (resp. binary, transversal) matroid is graphic (resp. binary, transversal).
\end{cor}

\noindent{\bf 4. Unique exchange property for bases with a view towards White's conjecture}

In this section, we investigate the preservation of White's conjecture under taking the expansion and contraction functors. First, we recall some notions and notations from \cite{Wh}.

Let $M$ be a matroid. Two sequences of bases $(A_1,\ldots,A_m)$ and $(B_1,\ldots,B_m)$ are \emph{compatible} if $A_1\cup\ldots\cup A_m=B_1\cup\ldots\cup B_m$. Let $(A_1,\ldots,A_m)$ be a sequence of bases of a matroid $M$ and let $1\leq r<s\leq m$ and $x\in A_r$. Set
$$E(x;A_r,A_s)=\{y\in A_s:(A_r\backslash x)\cup y,(A_s\backslash y)\cup x\in\B_M\}.$$
Let $y\in E(x;A_r,A_s)$. Then we say
$$(A_1,\ldots,A_{r-1},(A_r\backslash x)\cup y,A_{r+1},\ldots,A_{s-1},(A_s\backslash y)\cup x,A_{s+1},\ldots,A_m)$$
is obtained from $(A_1,\ldots,A_m)$ by a \emph{symmetric exchange}. For two sequences of bases of a matroid $M$, $(A_1,\ldots,A_m)\sim^1_M (B_1,\ldots,B_m)$ means that $(B_1,\ldots,B_m)$ is obtained from $(A_1,\ldots,A_m)$ by a symmetric exchange.  Also, $(A_1,\ldots,A_m)\sim^2_M (B_1,\ldots,B_m)$ implies to $(B_1,\ldots,B_m)$ is obtained from $(A_1,\ldots,A_m)$ by a symmetric exchange or a permutation of the order of the bases.

If $U\subset A_r$, we let
$$E(U;A_r,A_s)=\{V\subset A_s:(A_r\backslash U)\cup V,(A_s\backslash V)\cup U\in\B_M\}.$$
Let $V\in E(U;A_r,A_s)$. Then we say $(A_1,\ldots,A_{r-1},(A_r\backslash U)\cup V,A_{r+1},\ldots,A_{s-1},(A_s\backslash V)\cup U,A_{s+1},\ldots,A_m)$ is obtained from $(A_1,\ldots,A_m)$ by a \emph{symmetric subset exchange}. We write $(A_1,\ldots,A_m)\sim^3_M (B_1,\ldots,B_m)$ if $(B_1,\ldots,B_m)$ is obtained from $(A_1,\ldots,A_m)$ by a symmetric subset exchange.

For $i=1,2,3$, set $\TE(i)$ the class of matroids with the property that for every matroid $M\in\TE(i)$ and every two compatible sequences $(A_1,\ldots,A_m)$ and $(B_1,\ldots,B_m)$ of bases of $M$, there exist the sequences $(C_{j1},\ldots,C_{jm})$, for $j=1,\ldots,t$, of bases of $M$ such that
$$(A_1,\ldots,A_m)\sim^i_M (C_{11},\ldots,C_{1m})\sim^i_M\ldots\sim^i_M (C_{t1},\ldots,C_{tm})\sim^i_M (B_1,\ldots,B_m).$$
It is easy to check that $\TE(1)\subseteq\TE(2)\subseteq\TE(3)$. White conjectured that theses classes are equal to the class of all matroids \cite[Conjecture 12]{Wh}. We will say that a matroid $M$ satisfies White's conjecture if $M\in\TE(i)$, for all $i=1,2,3$.

\begin{thm}\label{TE}
For a matroid $M$ on $[n]$, $\alpha\in\NN^n$ and $i=1,2,3$ we have $$M\in\TE(i)\Leftrightarrow M^\alpha\in\TE(i).$$
\end{thm}

Before proving Theorem \ref{TE}, we present some auxiliary results.

\begin{lem}\label{pi0}
Let $M$ be a matroid on $[n]$ and $\alpha\in\NN^n$.
\begin{enumerate}[\upshape (i)]
  \item For any two compatible sequences $(A_1,\ldots,A_m)$ and $(B_1,\ldots,B_m)$ of bases of $M$, there exist two compatible sequences $(A'_1,\ldots,A'_m)$ and $(B'_1,\ldots,B'_m)$ of bases of $M^\alpha$ such that $\pi(A'_i)=A_i$ and $\pi(B'_i)=B_i$;
  \item If $(A'_1,\ldots,A'_m)$ and $(B'_1,\ldots,B'_m)$ are two compatible sequences of bases of $M^\alpha$ then $(\pi(A'_1),\ldots,\pi(A'_m))$ and $(\pi(B'_1),\ldots,\pi(B'_m))$ are two compatible sequences of bases of $M$.
\end{enumerate}
\end{lem}

\begin{lem}\label{expan to M}
Let $M$ be a matroid on $[n]$ and let $\alpha\in\NN^n$. For $B'_1,B'_2\in\B_{M^\alpha}$, if $(B'_1,B'_2)$ is obtained from $(A'_1,A'_2)$ by a symmetric subset exchange then
\begin{enumerate}[\upshape (i)]
  \item $\pi(B'_i)=\pi(A'_i)$ for $i=1,2$ or
  \item $(\pi(B'_1),\pi(B'_2))$ is obtained from $(\pi(A'_1),\pi(A'_2))$ by a symmetric subset exchange.
\end{enumerate}
\end{lem}

\begin{lem}\label{M to expan}
Let $M$ be a matroid on $[n]$ and let $\alpha\in\NN^n$. Suppose that $(B_1,B_2)$ is obtained from $(A_1,A_2)$ by a symmetric subset exchange. Let $(A'_1,A'_2)$ and $(B'_1,B'_2)$ be two compatible sequences of bases of $M^\alpha$ with $\pi(A'_i)=A_i$, $\pi(B'_i)=B_i$. Then $(B'_1,B'_2)$ is obtained from $(A'_1,A'_2)$ by a symmetric subset exchange.
\end{lem}

Now we prove Theorem \ref{TE} in three parts:

{\textbf{Membership in $\TE(3)$:}

Assume that $M\in\TE(3)$ and let $(A'_1,\ldots,A'_m)$ and $(B'_1,\ldots,B'_m)$ be two compatible sequences of bases of $M^\alpha$. Let $A_i=\pi(A'_i)$ and $B_i=\pi(B'_i)$ for $i=1,\ldots,m$. By Lemma \ref{pi0}(ii), $(A_1,\ldots,A_m)$ and $(B_1,\ldots,B_m)$ are compatible and, by the assumption,  $(B_1,\ldots,B_m)$ is obtained from $(A_1,\ldots,A_m)$ by a composition of symmetric subset exchanges. It follows from Lemma \ref{M to expan} that $(B'_1,\ldots,B'_m)$ is obtained from $(A'_1,\ldots,A'_m)$ by a composition of symmetric subset exchanges. Therefore $M^\alpha\in\TE(3)$.

Similarly, the converse direction follows from Lemmas \ref{pi0}(i) and \ref{expan to M}.

{\textbf{Membership in $\TE(2)$:}

Let $M\in\TE(2)$. Let $(A'_1,\ldots,A'_m)$ and $(B'_1,\ldots,B'_m)$ be two compatible sequences of $M^\alpha$. By Lemma \ref{pi0}(ii), $(\pi(A'_1),\ldots,\pi(A'_m))$ and $(\pi(B'_1),\ldots,\pi(B'_m))$ are two compatible sequences of bases of $M$ and so, by the assumption, $(\pi(B'_1),\ldots,\pi(B'_m))$ is obtained from $(\pi(A'_1),\ldots,\pi(A'_m))$ by a composition of symmetric exchanges and permutations of the order of the bases. It follows from Lemma \ref{M to expan} that $(B'_1,\ldots,B'_m)$ is obtained from $(A'_1,\ldots,A'_m)$ by a composition of symmetric exchanges and permutations of the order of the bases.

The converse direction obtains in a similar argument by using Lemmas \ref{pi0}(i) and \ref{expan to M}.

{\textbf{Membership in $\TE(1)$:}

\begin{lem}\label{TE1-TE2}
(\cite{LaMi}) Let $M$ be a matroid. Then $M\in\TE(1)$ if and only if $M\in\TE(2)$ and any pair $(A_2,A_1)$ of bases of $M$ is obtained from $(A_1,A_2)$ by a composition of symmetric exchanges.
\end{lem}

Since $M\in\TE(2)$ if and only if $M^\alpha\in\TE(2)$, it suffices to show that any pair $(A_2,A_1)$ of bases of $M$ is obtained from $(A_1,A_2)$ by a composition of symmetric exchanges if and only if any pair of bases of $M^\alpha$ has this property.

Suppose that any pair $(A_2,A_1)$ of bases of $M$ is obtained from $(A_1,A_2)$ by a composition of symmetric exchanges. Consider a pair $(B'_2,B'_1)$ of bases of $M^\alpha$. Let $B_i=\pi(B'_i)$. By the assumption, $(B_2,B_1)$ is obtained from $(B_1,B_2)$ by a composition of symmetric exchanges. Therefore
$$(B_1,B_2)\sim^1_M(C_{11},C_{12})\sim^1_M\ldots\sim^1_M (C_{t1},C_{t2})\sim^1_M (B_2,B_1).$$
By Lemma \ref{pi0}(i), one can choose the bases $C'_{ij}\in\B_{M^\alpha}$ with $\pi(C'_{ij})=C_{ij}$ such that $(B'_1,B'_2),(C'_{11},C'_{12}),\ldots,(C'_{t1},C'_{t2})$ are pairwise compatible. It follows from Lemma \ref{M to expan} that $(B'_2,B'_1)$ is obtained from $(B'_1,B'_2)$ by a composition of symmetric exchanges.

The converse follows from Lemmas \ref{pi0}(ii) and \ref{expan to M} in a similar argument.

\begin{cor}\label{contract white}
Let $M$ be a matroid on $[n]$ and let $\alpha\in\NN^n$. Then $M$ satisfies White's conjecture if and only if the contraction of $M$ does.
\end{cor}

\begin{rem}
Note that the class of contracted matroids is very smaller than the class of all matroids. It follows from Corollary \ref{contract white} that to test White's conjecture for a class of matroids it suffices to turn our attention to their contractions.
\end{rem}

\begin{cor}\label{partition white}
Every partition matroid satisfies White's conjecture.
\end{cor}
\begin{proof}
Let $M$ be a partition matroid. By Corollary \ref{cont part unif}, the contraction of $M$ is an uniform matroid. In view of Corollary \ref{contract white}, if we show that $U_{t,n}\in\TE(1)$ then the assertion is completed. It was shown in \cite{St} that the toric ideal of any uniform matroid is generated by quadratic binomials corresponding to symmetric exchanges and this is an algebraic meaning of the property $\TE(2)$. Hence $U_{t,n}\in\TE(2)$. Now suppose that $(A_2,A_1)$ is a pair of bases of $U_{t,n}$. Let $A_1=\{x_{i_1},\ldots,x_{i_t}\}$ and $A_2=\{x_{j_1},\ldots,x_{j_t}\}$. Let $i_{s+1}=j_{s+1},\ldots,i_t=j_t$ and $i_p\neq j_q$ for $p,q\leq s$. Then
\begin{flushleft}
$(A_1,A_2)\sim^1_{U_{t,n}}(\{x_{j_1},x_{i_2},\ldots,x_{i_t}\},\{x_{i_1},x_{j_2},\ldots,x_{j_t}\})\sim^1_{U_{t,n}}\ldots$
\end{flushleft}
\begin{flushright}
$\sim^1_{U_{t,n}} (\{x_{j_1},x_{j_2},\ldots,x_{j_{s-1}},x_{i_s},\ldots,x_{i_t}\},\{x_{i_1},x_{j_2},\ldots,x_{i_{s-1}},x_{j_s},\ldots,x_{j_t}\})\sim^1_{U_{t,n}} (A_2,A_1).$
\end{flushright}
It follows from Lemma \ref{TE1-TE2} that $M\in\TE(1)$, as desired.
\end{proof}

Let $M_1$ and $M_2$ be matroids on disjoint ground sets. The \emph{direct sum} of $M_1$ and $M_2$ is denoted by $M_1\oplus M_2$ and it is a matroid, by \cite[Proposition 4.2.12]{Ox}, on $\E_{M_1}\cup\E_{M_2}$ with the basis set
$$\B_{M_1\oplus M_2}=\{B_1\cup B_2:B_1\in\B_{M_1},B_2\in\B_{M_2}\}.$$

\begin{lem}\label{direct sum}
Let $M_1$ and $M_2$ be matroids on disjoint ground sets. If $M_1$ and $M_2$ satisfy White's conjecture then $M_1\oplus M_2$ does, too.
\end{lem}

The proofs of above lemma is easy and we leave them to the reader.

\begin{rem}\label{white rem}
(i)
Let $M$ be a matroid on $[n]=\{x_1,\ldots,x_n\}$ and let $z$ be disjoint from $x_i$'s. Then $M$ satisfies White's conjecture if and only if the matroid $N$ with the basis set $\{A\cup z:A\in\B_M\}$ does.

(ii) Let $M$ be a matroid on $[n]$. Consider $\{z_{ij}:i=1,\ldots,r,j=1,\ldots,s\}$ a set disjoint from $[n]$. Let $t\leq r,s$. Then the set
$$\B=\{A\cup (\cup^t_{k=1}z_{i_kj_k}):A\in\B_M, 1\leq i_1<\ldots<i_t\leq r, j_k\in [s]\}$$ is the basis set of a matroid $N$. In fact, $N=M\oplus M(\P)$ where $\P=\{A_1,\ldots,A_r\}$ and $A_i=\{z_{i1},\ldots,z_{is}\}$ for all $i$. Especially, if $M$ satisfies White's conjecture then it follows from Lemma \ref{direct sum} and Corollary \ref{partition white} that $N$ satisfies White's conjecture, too.
\end{rem}

Combining Theorem \ref{binary} with Theorem \ref{TE} we obtain

\begin{cor}
A binary matroid satisfies White's conjecture if and only if its contraction does.
\end{cor}

\begin{exam}
Consider the matroid $M$ on $[7]$ with the basis set
$$\B_M=\{\{x_1,x_2,x_3,x_5\},\{x_1,x_2,x_3,x_7\},\{x_1,x_2,x_5,x_6\},\{x_1,x_2,x_5,x_7\},\{x_1,x_2,x_6,x_7\},$$
$$\{x_1,x_3,x_4,x_5\},\{x_1,x_3,x_4,x_7\},\{x_1,x_4,x_5,x_6\},\{x_1,x_4,x_5,x_7\},\{x_1,x_4,x_6,x_7\},$$
$$\{x_2,x_3,x_5,x_7\},\{x_2,x_5,x_6,x_7\},\{x_3,x_4,x_5,x_7\},\{x_4,x_5,x_6,x_7\}\}.$$
It is easy to check that $M$ is binary. To see that $M$ satisfies White's conjecture, we first contract $M$. The contraction of $M$ is a binary matroid on $\{x_1,x_2,x_3,x_5,x_7\}$ with the basis set
$$\B_{\overline{M}}=\{\{x_1,x_2,x_3,x_5\},\{x_1,x_2,x_3,x_7\},\{x_1,x_2,x_5,x_7\},\{x_2,x_3,x_5,x_7\}\}.$$
On the other hand, $\B_{\overline{M}}$ can be obtained by adding $x_2$ to all bases of the uniform matroid $U_{3,4}$ with the ground set $\{x_1,x_3,x_5,x_7\}$. It follows from Remark \ref{white rem}(i) that $\overline{M}$ satisfies White's conjecture and so $M$ satisfies White's conjecture, too.
\end{exam}

\footnotesize{

}


\begin{thebibliography}{99}

\bibitem{Bl} J. Blasiak, \textit{The toric ideal of a graphic matroid is generated by quadrics}, Combinatorica \textbf{28} (3) (2008) 283-297.

\bibitem{Bo} J. Bonin, \textit{Basis-exchange properties of sparse paving matroids}, Adv. in Appl. Math. \textbf{50} (2013) no. 1, 6-15.

\bibitem{Br} R. A. Brualdi, \textit{Transversal matroids, in: Combinatorial Geometries}, N. White, ed. (Cambridge Univ. Press, Cambridge (1987)) 72-97.

\bibitem{FrHaVi} C. A. Francisco; H. T. H\`{a}, A. Van Tuyl,  \textit{Coloring of hypergraphs, perfect graphs and associated primes of powers of monomial ideals.} J. Algebra \textbf{331} (2011), 224-242.

\bibitem{Go} M. C. Golumbic, \textit{Algorithmic graph theory and perfect graphs}, second edition, Annals of Discrete Mathematics \textbf{57}, Elsevier Science B.V., Amsterdam, 2004.

\bibitem{Ha} P. Hall, \textit{On representatives of subsets.} J. London Math. Soc. \textbf{10} (1935) 26-30.

\bibitem{Ho} W. Hochst\"{a}ttler, \textit{The toric ideal of a cographic matroid is generated by quadrics}, Technical Report feu-dmo023.10, FernUniversitat in Hagen Fakult\"{a}t f\"{u}r Mathematik und Informatik.

\bibitem{Jo} K.D. Joshi, \textit{Applied Discrete Structures}, New Age International(P) limited, publisher (1997).

\bibitem{Ka} K. Kashiwabara, \textit{The toric ideal of a matroid of rank 3 is generated by quadrics}, Electron. J. Combin. \textbf{17} (2010) RP28, 12pp.

\bibitem{KaOkUn} K. Kashiwabara, Y. Okamoto, T. Uno, \textit{Matroid representation of clique complexes}, Discrete Appl. Math. \textbf{155} (2007) 1910-1929.

\bibitem{LaMi} M. Laso\'{n}, M. Michalek, \textit{On the toric ideal of a matroid}, Adv. Math. \textbf{259} (2014) 1-12.

\bibitem{Ma} S. Mac Lane, \textit{Some interpretations of abstract linear independence in terms of projective geometry}. American Journal of Mathematics \textbf{58}(1), 236-240 (1936).

\bibitem{MaReVi} J. Mart\'{\i}nez-Bernal, C. Renter\'{\i}a, and R. H. Villarreal, \textit{Combinatorics of symbolic Rees algebras of edge ideals of clutters}, Commutative Algebra and its Connections to Geometry, Contemporary Mathematics, \textbf{555}, eds. Corso A., Polini C., (2011) 151-164

\bibitem{Ox} J.G. Oxley, \textit{Matroid Theory}, Oxford Science Publications, Oxford University Press, Oxford, 1992.

\bibitem{RaYa} R. Rahmati-Asghar, S. Yassemi, \textit{The behaviors of expansion functor on monomial ideals and toric rings}, Communications in Algebra, \textbf{44}, (2016), no. 9, 3874-3889.

\bibitem{Scr} A. Schrijver, \textit{Combinatorial Optimization, Algorithms and Combinatorics}, \textbf{24}, Springer-Verlag, Berlin, 2003.

\bibitem{Sc} J. Schweig, \textit{Toric ideals of lattice path matroids and polymatroids}, Journal of Pure and Applied Algebra, \textbf{215}(11) (2011) 2660-2665.

\bibitem{St} B. Sturmfels, \textit{Gr\"{o}bner Bases and Convex Polytopes}, American Mathematical Sociey, University Lecture Series, Vol. 8, Providence, RI, 1995.

\bibitem{We} D. J. A. Welsh, \textit{Matroid Theory}, Academic, New York, 1976.

\bibitem{Wh} N. White, \textit{A unique exchange property for bases}, Linear Algebra and its App. \textbf{31} (1980) 81-91.

\bibitem{Whn} H. Whitney, \textit{On the abstract properties of linear dependence}, Amer. J. Math. \textbf{57} (1935) 509-533.

\end{thebibliography}
\end{document}